\documentclass[11pt,reqno]{amsart} % font size - 11pt, right equation numbers, type - AMS Article

\usepackage{amsmath,amsthm,amssymb,mathtools,mathrsfs,amsfonts,amscd}
\usepackage{comment,fullpage}
\usepackage{braket}
\usepackage{caption}
\usepackage{times}
\usepackage[T1]{fontenc}
\usepackage{latexsym}
\usepackage[dvips]{graphics}
\usepackage{epsfig}
\input amssym.def
\input amssym.tex
\usepackage{hyperref}
\usepackage{url}
\usepackage{color}
\usepackage{breakurl}

\usepackage{tikz}
\usepackage{floatrow}

\newcommand{\bburl}[1]{\textcolor{blue}{\url{#1}}}

\numberwithin{equation}{section}

\newtheorem{thm}{Theorem}[section]

\theoremstyle{plain}

\newtheorem{corollary}[thm]{Corollary}
\newtheorem{definition}[thm]{Definition}

\newtheorem{lemma}[thm]{Lemma}
\newtheorem{proposition}[thm]{Proposition}
\newtheorem{theorem}[thm]{Theorem}
\newtheorem{conjecture}[thm]{Conjecture}

\newcommand{\hr}[1]{\href{#1}{\url{#1}}}

\title{The Fibonacci Quilt Game}

\author{Steven J. Miller}
\email{\textcolor{blue}{\href{mailto:sjm1@williams.edu}{sjm1@williams.edu}},  \textcolor{blue}{\href{Steven.Miller.MC.96@aya.yale.edu}{Steven.Miller.MC.96@aya.yale.edu}}}
\address{Department of Mathematics and Statistics, Williams College, Williamstown, MA 01267}

\author{Alexandra Newlon}
\email{\textcolor{blue}{\href{mailto:aknewlon@gmail.com}{aknewlon@gmail.com}}}
\address{Department of Mathematics, Colgate University, Hamilton, NY 13346}

\thanks{This work was partially supported by NSF grants DMS1659037 and DMS1561945 and by Williams College. We thank the participants of the 2019 Williams SMALL REU for helpful conversations.}


\date{\today}

\begin{document}
\maketitle

\begin{abstract}
Zeckendorf \cite{Ze} proved that every positive integer can be expressed as the sum of non-consecutive Fibonacci numbers. This theorem inspired a beautiful game, the Zeckendorf Game \cite{BEFM1}. Two players begin with $n \ 1$'s and take turns applying rules inspired by the Fibonacci recurrence, $F_{n+1} = F_n + F_{n-1}$, until a decomposition without consecutive terms is reached; whoever makes the last move wins. We look at a game resulting from a generalization of the Fibonacci numbers, the Fibonacci Quilt sequence \cite{CFHMN}. These arise from the two-dimensional geometric property of tiling the plane through the Fibonacci spiral. Beginning with 1 in the center, we place integers in the squares of the spiral such that each square contains the smallest positive integer that does not have a decomposition as the sum of previous terms that do not share a wall. This sequence eventually follows two recurrence relations, allowing us to construct a variation on the Zeckendorf Game, the Fibonacci Quilt Game. While some properties of the Fibonaccis are inherited by this sequence, the nature of its recurrence leads to others, such as Zeckendorf's theorem, no longer holding; it is thus of interest to investigate the generalization of the game in this setting to see which behaviors persist. We prove, similar to the original game, that this game also always terminates in a legal decomposition, give a lower bound on game lengths, show that depending on strategies the length of the game can vary and either player could win, and give a conjecture on the length of a random game.
\end{abstract}

\tableofcontents

%%%%%%%%%%%%%%%%
% INTRODUCTION %
%%%%%%%%%%%%%%%%
%%%%%%%%%%%%%%%%%%%%%%%%%%%%%%%%%%%%%%%%%%%%%%%%%%%%%%%%%%%%%%%%%%%%%%%%%%%%%%%%%%%%%%%%%%%%%%%%%%%%%%%%%%%%%%%%%%%%%%%%%%%%%%%%%%%%%%%%%%%%%%%%%%%%%%%%%%
%%%%%%%%%%%%%%%%%%%%%%%%%%%%%%%%%%%%%%%%%%%%%%%%%%%%%%%%%%%%%%%%%%%%%%%%%%%%%%%%%%%%%%%%%%%%%%%%%%%%%%%%%%%%%%%%%%%%%%%%%%%%%%%%%%%%%%%%%%%%%%%%%%%%%%%%%%
%%%%%%%%%%%%%%%%%%%%%%%%%%%%%%%%%%%%%%%%%%%%%%%%%%%%%%%%%%%%%%%%%%%%%%%%%%%%%%%%%%%%%%%%%%%%%%%%%%%%%%%%%%%%%%%%%%%%%%%%%%%%%%%%%%%%%%%%%%%%%%%%%%%%%%%%%%
%%%%%%%%%%%%%%%%%%%%%%%%%%%%%%%%%%%%%%%%%%%%%%%%%%%%%%%%%%%%%%%%%%%%%%%%%%%%%%%%%%%%%%%%%%%%%%%%%%%%%%%%%%%%%%%%%%%%%%%%%%%%%%%%%%%%%%%%%%%%%%%%%%%%%%%%%%

\section{Introduction}\label{sec:intro}

%%%%%%%%%%%%%%%%%%%%%%
% HISTORY OF PROBLEM %
%%%%%%%%%%%%%%%%%%%%%%
%%%%%%%%%%%%%%%%%%%%%%%%%%%%%%%%%%%%%%%%%%%%%%%%%%%%%%%%%%%%%%%%%%%%%%%%%
%%%%%%%%%%%%%%%%%%%%%%%%%%%%%%%%%%%%%%%%%%%%%%%%%%%%%%%%%%%%%%%%%%%%%%%%%
\subsection{History}\label{sec:history}
The Fibonacci numbers are one of the most famous, and beautiful, sequences of all time; appearing throughout mathematics and nature \cite{Kos}. Zeckendorf \cite{Ze} proved that every positive integer has a unique representation as a sum of non-consecutive Fibonacci numbers, which are defined by $F_1 = 1, F_2 = 2$ and $F_{n+1} = F_n + F_{n-1}$; conversely, an equivalent definition of the Fibonacci numbers is the unique sequence of integers such that every positive integers can be uniquely written as a sum of non-consecutive terms. Here, we set the initial conditions $F_1 = 1$ and $F_2 = 2$ rather then $F_1 = F_2 = 1$ to preserve uniqueness. This is the first of many interplays between notions of legal decomposition and definitions of a sequence, expanded to a large class of linear recurrences (see \cite{Ho, Ke, MW1, MW2} for examples).

%%%%%%%%%%%%%%%%%%%%%%%
% The Zeckendorf Game %
%%%%%%%%%%%%%%%%%%%%%%%

\subsubsection{The Zeckendorf Game}\label{sec:zeckgame}
We can use these notions of legal decomposition to create interesting games. The first, the Zeckendorf game, was defined based on the recurrence relation of the Fibonacci sequence $\{F_n\}$. We briefly summarize the results from \cite{BEFM1, BEFM2}.

We first introduce some notation. Let $\{F_1^n\}$ denote $n$ copies of $F_1$, and in general $\{F_i^n\}$ denote $n$ copies of $F_i$; as we never raise Fibonacci numbers to a power there should be no confusion as to what is meant. For example, $\{F_1^3 \land F_4^2 \land F_5^1\}$ would be three copies of $F_1 = 1$, two copies of $F_4 = 5$, and one copy of $F_5 = 8$. For simplicity, moving forward we omit exponents of 1, so $\{F_i\} = \{F_i^1\}$.

\begin{definition}[The Two Player Zeckendorf Game]\label{def:zeckgame}
     At the beginning of the game, there is an unordered list of $n$ 1's. Let $F_1 = 1$, $F_2 =  2$, and $F_{i+1} = F_i + F_{i-1}$; therefore the initial list is $\{F_1^n\}$. On each turn, a player can do one of the following moves.
    \begin{enumerate}
        \item If the list contains two consecutive Fibonacci numbers, $F_{i-1}$ and $F_i$, then a player can remove these and replace with $F_{i+1}$. We denote this move $\{F_{i-1} \land F_i \rightarrow F_{i+1}\}$.
        \item If the list has two (or more) of the same Fibonacci number, $F_i$, then
        \begin{enumerate}
            \item if $i = 1$, a player can change two $F_1$'s to $F_2$, denoted by $\{F_1^2 \rightarrow F_2\}$,
            \item if $i = 2$, a player can change two $F_2$'s to $F_1$ and $F_3$, denoted by $\{F_2^2 \rightarrow F_1  \land  F_3\}$, and
            \item if $i$ $\geq 3$, a player can change two $F_i$'s to $F_{i-2}$ and $F_{i+1}$, denoted by $\{F_i^2 \rightarrow F_{i-2}  \land  F_{i+1}\}$.
        \end{enumerate}
    \end{enumerate}
    The players alternative moving.  The game ends when no more moves are possible, and the last person to move wins.
\end{definition}

Baird-Smith, Epstein, Flint, and Miller \cite{BEFM1, BEFM2} proved that this game always terminates in a finite number of moves in the Zeckendorf decomposition of $n$, and then bounded the game length. One of the key ingredients in their proof is that there is no decomposition involving sums of Fibonacci numbers with fewer summands than the Zeckendorf decomposition; this is proved using a monovariant related to the number and indices of each term, and has been generalized to many other sequences \cite{CHHMPT}.

\begin{theorem}\label{thm:zeckterm}
The shortest game reaches the Zeckendorf decomposition in $n - Z(n)$ moves, where $Z(n)$ is the number of terms in the Zeckendorf decomposition of $n$. The longest game is bounded by $i \times n$, where $i$ is the index of the largest Fibonacci number less than or equal to $n$.
\end{theorem}

Since there is a large range between the lower and upper bounds, they also conjectured on the length of a random game.

\begin{conjecture}\label{con:zeckgauss}
As $n$ goes  to  infinity,  the  number  of  moves  in  a  random game, when all legal moves are equally likely, converges to a Gaussian.
\end{conjecture}

Finally, they found that for $n > 2$,  Player 2 has the winning strategy; interestingly, however, the proof is non-constructive. While it is known that Player 2 can win, it is not known how they should play.

In this paper we generalize their results by replacing the Fibonacci numbers with the Fibonacci Quilt. We define this sequence in the next section, and explain why this is an interesting extension.

%%%%%%%%%%%%%%%%%%%%%%%%%%%%%%%%
% The Fibonacci Quilt Sequence %
%%%%%%%%%%%%%%%%%%%%%%%%%%%%%%%%
\subsubsection{The Fibonacci Quilt Sequence}\label{sec:fibq}

Previous work extended Zeckendorf's theorem to a wide class of recurrence relations (see \cite{Ho, Ke}), and has extensively studied the behavior of these decompositions. Lekkerkerker \cite{Lek} proved the mean number of terms needed in a decomposition grows linearly with the largest index in the decomposition, and Kolo$\breve{{\rm g}}$lu, Kopp, Miller, and Wang \cite{KKMW, MW1, MW2} expanded this to show the distribution of the number of terms in a decomposition of $n$ between two consecutive terms of the sequence is Gaussian. This work, however, is done on a class of recurrences called PLRS's (for Positive Linear Recurrence Relations). Briefly, these are fixed depth constant coefficient linear recurrences where the coefficients are non-negative integers, the first coefficient in the recurrence is positive, and the initial conditions are chosen appropriately; if the first coefficient is not positive then different behavior can happen, in particular unique decomposition is often lost.

When looking to expand this work further, Catral, Ford, Harris, Miller, and Nelson \cite{CFHMN} wanted to explore new interesting patterns. The Fibonacci Quilt sequence arises from a 2-dimensional construction and is eventually dictated by a recurrence relation with first coefficient zero; thus the previous work is not applicable here and while some properties are the same, we will see others are different.

Recall the alternative definition of the Fibonacci numbers stated above; they are the unique sequence of integers such that every positive integers can be uniquely written as a sum of non-consecutive terms. The Fibonacci Quilt sequence is similarly defined on the Fibonacci spiral, where each term added is the smallest positive integer that cannot be expressed as the sum of non-adjacent previous terms.

\begin{figure}[!h]
    \begin{floatrow}
        \ffigbox{\includegraphics[width=2.5in]{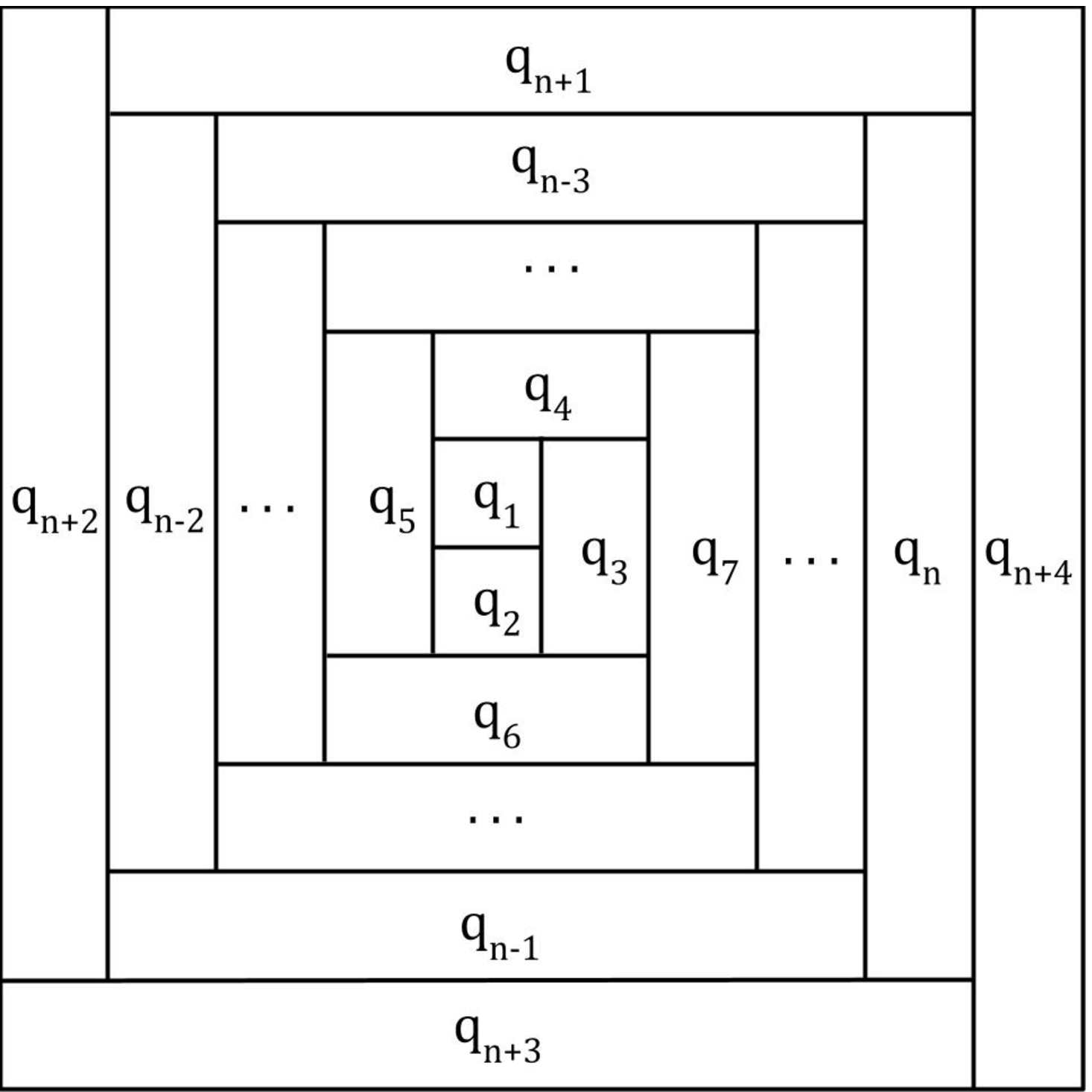}}{\caption{Log Cabin Quilt Pattern}\label{fig:lcqp}}
        \ffigbox{\includegraphics[width=2.5in]{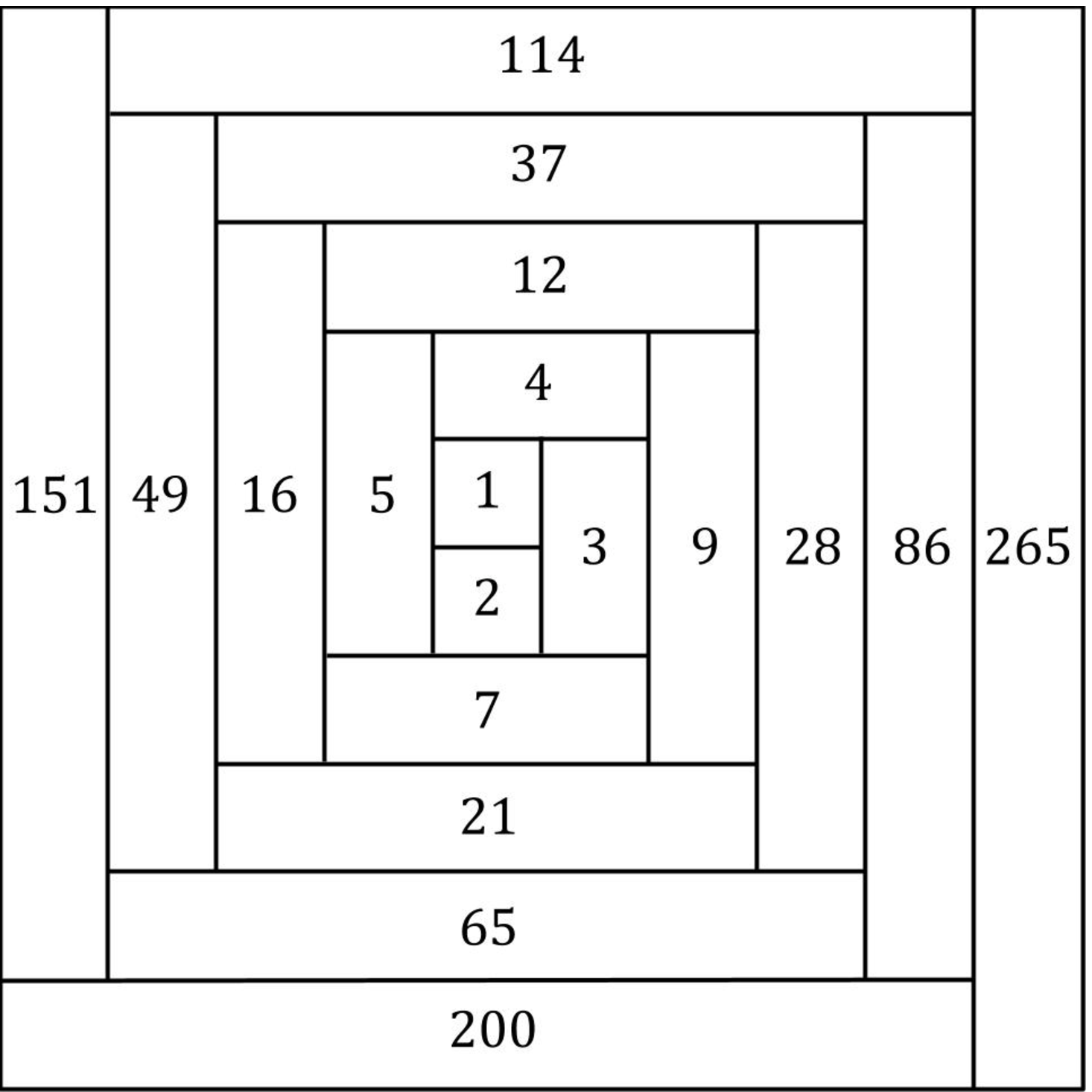}}{\caption{The Fibonacci Quilt Sequence}\label{fig:fqs}}
    \end{floatrow}
\end{figure}

The spiral is known in quilting communities as the Log Cabin quilt pattern, giving this sequence its name.  To construct the sequence begin with 1 in the $q_1$ position, then spiral out adding the smallest positive integer that cannot be expressed as the sum of non-adjacent previous terms; two terms are adjacent if they share part of a wall. In order to display more terms, we adjust the size of the spiral and make all horizontal distances 1 unit, and have the vertical distances the appropriate size from the spiral. For example, the first positive integer we do not add is $6$, since it can be expressed as $2 + 4$. To formalize our definition of this sequence, we must first formalize what it means to be expressed as the sum of non-adjacent terms.

\begin{definition}[FQ-legal decomposition]\label{def:fqlegal} \cite{CFHMN}
    Let an increasing sequence of positive integers $\{q_i\}_{i=1}^\infty$ be given. We declare a decomposition of an integer
    \begin{equation}
        m\ =\ q_{\ell_1}  +  q_{\ell_2} + \cdots + q_{\ell_t}
    \end{equation}
    (where $q_{\ell_i} > q_{\ell_{i+1}}$) to be an FQ-legal decomposition if for all $i$ and $j$ we have  $|\ell_i - \ell_j| \neq 0,1,3,4$ and $\{1,3\}\not\subset\{\ell_1,\ell_2,\dots,\ell_t\}$.
\end{definition}

To better understand this definition see Figure \ref{fig:lcqp}. Looking at terms less than or equal to itself, $q_{n+4}$ is adjacent to itself, $q_{n}$, $q_{n+1}$, and $q_{n+3}$, thus if any of these were present in the decomposition of some $m$, then $q_{n+4}$ could not be present without violating this definition. Further $q_1$ and $q_3$ cannot both be present since they are adjacent in the center of the quilt, although no other $q_n$ and $q_{n+2}$ are. Interestingly, unlike the Fibonacci numbers, not all integers have a unique FQ-legal decompositions; for example, $8 = 1 + 7 = 3 + 5$ are both FQ-legal decompositions of 8.

With this we can now formalize the definition of the Fibonacci Quilt Sequence.

\begin{definition}[Fibonacci Quilt Sequence]\label{def:fqs} \cite{CFHMN}
    An increasing sequence of positive integers $\{q_i\}_{i=1}^\infty$ is called the Fibonacci Quilt sequence if every $q_i$ ($i\geq1$) is the smallest positive integer that does not have an FQ-legal decomposition using the elements $\{q_1, \dots, q_{i-1}\}$.
\end{definition}

While this definition is mathematically precise, in practice it is still computation and time intensive to determine $q_n$, even knowing $q_1, q_2, \dots, q_{n-1}$. Luckily after a short time, the behavior of the sequence can be explained by recurrence relations.

\begin{theorem}[Recurrence Relations]\label{thm:rrs} \cite{CFHMN}
    Let $q_n$ denote the $n$\textsuperscript{{\rm th}} term in the Fibonacci Quilt. Then
    \begin{eqnarray}\label{eq:fibquiltrelations}
        q_{n+1} & \ =\ & q_n + q_{n-4} \ \ \text{for } n \geq 6, \nonumber\\
        q_{n+1} & \ =\ & q_{n-1} + q_{n-2} \ \ \text{for } n \geq 5, \nonumber\\
        \sum_{i=1}^n q_i & \ = \ & q_{n+5} - 6.
    \end{eqnarray}
\end{theorem}

Note that the recurrence relation of \emph{minimal} length is the second above, and as the leading coefficient there (the $q_n$ term) is zero we do not have a PLRS.\footnote{The first recurrence relation is a PLRS, but the initial conditions for the Fibonacci Quilt come from the second relation, and thus while this could generate a PLRS, it does not generate a PLRS for our situation due to the different initial conditions.}

From these recurrence relations we can build our game, which we describe in the next section. Similar to the Zeckendorf Game, the rules follow from the recurrence relations that describe the sequence, however new interesting features arise from the non-uniqueness of decompositions, and the different behavior of the quilt at the center coming from its 2-dimensional definition.

%%%%%%%%%%%%%%%%
% MAIN RESULTS %
%%%%%%%%%%%%%%%%
%%%%%%%%%%%%%%%%%%%%%%%%%%%%%%%%%%%%%%%%%%%%%%%%%%%%%%%%%%%%%%%%%%%%%%%%%
%%%%%%%%%%%%%%%%%%%%%%%%%%%%%%%%%%%%%%%%%%%%%%%%%%%%%%%%%%%%%%%%%%%%%%%%%

\subsection{Main Results}\label{sec:mainresults}
Although the Fibonacci Quilt Game is adapted from the Zeckendorf Game, it requires many more moves. This is firstly because in a Zeckendorf Decomposition there are only two criteria required for legality: no duplicate terms, and no consecutive terms. The Fibonacci Quilt Game requires five, which are direct results of Definition \ref{def:fqlegal}: no duplicate terms, no consecutive terms, no terms of distance 3 apart, no terms of distance 4 apart, and 1 and 3 cannot both be present. Each of these requirements creates a new rule.

Each of these rules also requires many base case rules, this is due to the construction of the Fibonacci Quilt sequence; the quilt behaves differently in the center causing the recurrence relations in \eqref{eq:fibquiltrelations} to begin later. The base rules are largely intuitive, e.g., $1 \land 2 = 3$ not $4$, as it would in the general rule. The general rules arise from how the recurrence relation combines terms. The most interesting is Rule (2a) below, which states a certain move can only be done if no other moves are available; without this addition the game need not terminate. It is similar in spirit to the Greedy-6 decomposition from \cite{CFHMN} (which leads to unique decompositions). We will see later that we can associate an almost monovariant to the game; it breaks down for Rule (2a), but our requirements imply that this rule is used at most once, and thus our quantity is effectively as good as a true monovariant.

The notation used for the Fibonacci Quilt Game is similar to that of the Zeckendorf Game. Let $\{1^n\}$ or $\{q_1^n\}$ be $n$ copies of 1, and in general $\{q_i^n\}$ be $n$ copies of $q_i$. For example, $\{q_1^3 \land q_3^2 \land q_4^1 \}$ would be three copies of 1, two copies of 3, and one copy of 4.

\begin{definition}[The Two Player Fibonacci Quilt Game] \label{def:fibqgame}
At the beginning of the game there is an unordered list of $n$ 1's. Let $q_1 = 1$, $q_2 = 2$, $q_3 = 3$, $q_4 = 4$, and, for $i \geq 5$, $q_i = q_{i-3} + q_{i-2}$; therefore the initial list is $\{q_1^{n}\}$. Players alternate turns, and on each turn can make one of the following moves.
    \begin{enumerate}
        \item If the list contains two consecutive Fibonacci Quilt terms, $q_i$ and $q_{i+1}$, then
        \begin{enumerate}
            \item if $i = 1$, a player can change $q_1$ and $q_2$ to $q_3$, denoted $\{q_1 \land q_2 \rightarrow q_3\}$, and
            \item if $i \geq 2$, a player can change $q_i$ and $q_{i+1}$ to $q_{i+3}$, denoted $\{q_i \land q_{i+1} \rightarrow q_{i+3}\}$.
        \end{enumerate}
        \item If the list contains two Fibonacci Quilt terms of distance 4 apart, $q_i$ and $q_{i+4}$, then
        \begin{enumerate}
            \item if $i = 1$, and no other moves are possible, a player can change $q_1$ and $q_5$ to $q_2$ and $q_4$, denoted $\{q_1  \land q_5 \rightarrow q_2  \land q_4 \}$, and
            \item if $i \geq 2$, a player can change $q_i$ and $q_{i+4}$ to $q_{i+5}$, denoted $\{q_i \land  q_{i+4} \rightarrow q_{i+5}\}$.
        \end{enumerate}
        \item If the list contains two (or more) of the same Fibonacci Quilt term, $q_i$, then
        \begin{enumerate}
            \item if $i = 1$, a player can change $q_1$ and $q_1$ to $q_2$, denoted $\{q_1^2 \rightarrow q_2\}$,
            \item if $i = 2$, a player can change $q_2$ and $q_2$ to $q_4$, denoted $\{q_2^2 \rightarrow q_4\}$,
            \item if $i = 3$, a player can change $q_3$ and $q_3$ to $q_2$ and $q_4$, denoted $\{q_3^2 \rightarrow q_2  \land  q_4\}$,
            \item if $i = 4$, a player can choose to change $q_4$ and $q_4$ to $q_1$ and $q_6$ \textbf{or} $q_3$ and $q_5$, denoted $\{q_4^2 \rightarrow q_1 \land  q_6\}$ and $\{q_4^2 \rightarrow q_3  \land  q_5\}$ respectively,
            \item if $i = 5$, a player can change $q_5$ and $q_5$ to $q_1$ and $q_7$, denoted $\{q_5^2 \rightarrow q_1 \land  q_7\}$,
            \item if $i = 6$, a player can choose to change $q_6$ and $q_6$ to $q_2$ and $q_8$ \textbf{or} $q_3$ and $q_7$, denoted $\{q_6^2 \rightarrow q_2 \land  q_8\}$ and $\{q_6^2 \rightarrow q_3  \land q_7\}$ respectively, and
            \item if $i \geq 7$, a player can change $q_i$ and $q_i$ to $q_{i-5}$ and $q_{i+2}$, denoted $\{q_i^2 \rightarrow q_{i-5}  \land q_{i+2}\}$.
        \end{enumerate}
        \item If the list contains two Fibonacci Quilt terms of distance 3 apart, $q_i$ and $q_{i+3}$, then
        \begin{enumerate}
            \item if $i = 1,2$, a player can change $q_i$ and $q_{i+3}$ to $q_{i+4}$, denoted $\{q_i  \land  q_{i+3} \rightarrow q_{i+4}\}$,
            \item if $i = 3$, a player can change $q_3$ and $q_6$ to $q_1$ and $q_7$, denoted $\{q_3  \land  q_6 \rightarrow q_1  \land q_7\}$,
            \item if $i = 4,5$, a player can change $q_i$ and $q_{i+3}$ to $q_1$ and $q_{i+4}$, denoted $\{q_i  \land \ q_{i+3} \rightarrow q_1  \land  q_{i+4}\}$,
            \item if $i = 6$, a player can change $q_6$ and $q_9$ to $q_2$ and $q_{10}$, denoted $\{q_6  \land  q_9 \rightarrow q_2  \land  q_{10}\}$, and
            \item if $i \geq 7$, a player can change $q_i$ and $q_{i+3}$ to $q_{i-5}$ and $q_{i+4}$, denoted $\{q_i  \land  q_{i+3} \rightarrow q_{i-5}  \land  q_{i+4}\}$.
        \end{enumerate}
        \item If the list contains $q_1$ and $q_3$, a player can change $q_1$ and $q_3$ to $q_4$, denoted $\{q_1  \land  q_3 \rightarrow q_4\}$.
    \end{enumerate}
    The game ends when there are no possible moves, and whomever made the last move wins.
\end{definition}

The moves for this game may seem random, but they are a direct result of the recurrence relations stated in Theorem \ref{thm:rrs}, and Definition \ref{def:fqlegal} (FQ-legal decomposition). Each rule when applied takes two terms which could not be in a legal decomposition together and changes them to a legal term or pair of terms. For example, Rule 1 takes terms which are distance 1 apart, or $q_i$ and $q_j$ such that $j-i=1$, and changes them to a single term.

There are many cases for each rule because the Fibonacci Quilt sequence does not follow the recurrence relations of \eqref{eq:fibquiltrelations} at the very beginning, and thus the same rules cannot be applied there. Each base rule is created to change terms to a legal term or pair of terms while preserving that the sum of the list is $n$.

Two important things to note are Rule (2a) and Rules (3d) and (3f). Rule (2a) can only be applied when no other moves are possible, that is the list contains no other illegal pairs besides ($q_1,q_5$). For Rules (3d) and (3f) the player has two options, this is because the Fibonacci Quilt Sequence lacks uniqueness, so $n=8$ can be decomposed into $1+7$ or $3+5$, both of which are legal. We will show later that for $i\geq 7$, $2q_i$ can only be decomposed into two terms legally by Rule (3g).

With this construction we first show that it is well-defined, and then study the length of a game.

\begin{theorem}\label{thm:terminates}
    Every game terminates in a finite number of moves at a FQ-legal decomposition.
\end{theorem}

Knowing that the game terminates, we can also ask how quickly it can end. We give a result for the shortest game; as we are able to associate a monovariant to the game, by looking at the smallest change possible for the summands that can be in play (i.e., we can never have a summand larger $q_m > n$) one could isolate an upper bound as well.

\begin{theorem}\label{thm:shortest}
    The shortest game on $n$ arrives at a FQ-legal decomposition in $n - L(n)$ moves, where $L(n)$ is the maximum number of terms in a FQ-legal decomposition of $n$.
\end{theorem}

We can also look at the length of a completely random game.

\begin{conjecture}\label{con:fqgauss}
    As $n$ goes  to  infinity,  the  number  of  moves  in  a  random game  decomposing $n$ into it's Zeckendorf expansion, when all legal moves are equally likely, converges to a Gaussian.
\end{conjecture}

The next section will provide proofs for each of these theorems, starting with key lemmas and building up, as well as evidence to support our conjecture. Finally we will pose some questions we still hope to answer as well as possible future work.

%%%%%%%%%%%%%%%%%%%%%%%%%%%%
% THE FIBONACCI QUILT GAME %
%%%%%%%%%%%%%%%%%%%%%%%%%%%%
%%%%%%%%%%%%%%%%%%%%%%%%%%%%%%%%%%%%%%%%%%%%%%%%%%%%%%%%%%%%%%%%%%%%%%%%%%%%%%%%%%%%%%%%%%%%%%%%%%%%%%%%%%%%%%%%%%%%%%%%%%%%%%%%%%%%%%%%%%%%%%%%%%%%%%%%%%
%%%%%%%%%%%%%%%%%%%%%%%%%%%%%%%%%%%%%%%%%%%%%%%%%%%%%%%%%%%%%%%%%%%%%%%%%%%%%%%%%%%%%%%%%%%%%%%%%%%%%%%%%%%%%%%%%%%%%%%%%%%%%%%%%%%%%%%%%%%%%%%%%%%%%%%%%%
%%%%%%%%%%%%%%%%%%%%%%%%%%%%%%%%%%%%%%%%%%%%%%%%%%%%%%%%%%%%%%%%%%%%%%%%%%%%%%%%%%%%%%%%%%%%%%%%%%%%%%%%%%%%%%%%%%%%%%%%%%%%%%%%%%%%%%%%%%%%%%%%%%%%%%%%%%
%%%%%%%%%%%%%%%%%%%%%%%%%%%%%%%%%%%%%%%%%%%%%%%%%%%%%%%%%%%%%%%%%%%%%%%%%%%%%%%%%%%%%%%%%%%%%%%%%%%%%%%%%%%%%%%%%%%%%%%%%%%%%%%%%%%%%%%%%%%%%%%%%%%%%%%%%%

\section{The Fibonacci Quilt Game}\label{sec:fibquiltgame}

%%%%%%%%%%%%%%%%%%%%%%%%%%%%%%%%%%%%%%%%%%%%%%%%%%%%%%%%%%%%%%%%%%%%%%%%%
%%%%%%%%%%%%%%%%%%%%%%%%%%%%%%%%%%%%%%%%%%%%%%%%%%%%%%%%%%%%%%%%%%%%%%%%%
\subsection{The Game is Playable} The game as stated in Definition \ref{def:fibqgame} has two rules where the player can choose between two possible decompositions. Specifically, if there are two $q_4$, the player may choose to make the move $\{q_4^2 \rightarrow q_1 \land q_6\}$ or the move $\{q_4^2 \rightarrow q_3 \land q_5\}$, and if there are two $q_6$, the player may choose to make the move $\{q_6^2 \rightarrow q_2 \land q_8\}$ or the move $\{q_6^2 \rightarrow q_3 \land q_7\}$. To ensure that the given definition of the game encompasses all possible moves we first verify that for $n\geq7$, $\{q_i^2 \rightarrow q_{i-5} \land q_{i+2}\}$ is the only possible move.

\begin{proposition}\label{prop:split}
 Given $q_i^2$ for $i \geq 7$, the only legal way to decompose $q_i^2$ into two terms is $\{q_i^2 \rightarrow q_{i-5} \land q_{i+2}\}$.
\end{proposition}

\begin{proof}
Suppose $2q_n = q_i + q_j$, and without loss of generality let $i>j$. We know $2q_n = q_n + q_n < q_n + q_{n+1} = q_{n+3}$, so $i<n+3$.

If $i = n$, then $j = n$ gives us an illegal decomposition.

If $i < n$, then $j < n$, but the Fibonacci Quilt sequence is strictly increasing, so $2q_n \not= q_i + q_j$ for $i,j < n$.

So $i = n + 1$ or $i = n + 2$. If $i = n + 2$ then we get the known solution $2q_n = q_{n-5} + q_{n+2}$. If $i=n+2$, $j=n-5$ is the unique solution to this addition. So we must verify there are no legal decompositions for $i = n + 1$. If $j = n - 3$, then $q_i + q_j = q_{n+2} < q_{n+2} + q_{n-5} = 2q_n$, so $ n - 2 \leq j \leq n$. $j \in \{n-2,n\}$ gives an illegal decomposition with $i = n+1$. So the only possible case is $j = n - 1$. But $2q_n < q_{n-9} + 2q_n = q_{n-1} + q_{n-4} + q_n = q_{n-1} + q_{n+1}$ by applying Rules (4e) and (2b). Thus there is no value of $j$ for $i = n + 1$ and $2q_n = q_{n-5} + q_{n+2}$ is the only legal decomposition using two terms.
\end{proof}

Now that we have established this we may prove Theorem \ref{thm:terminates}, starting with a few crucial lemmas. Our proof strategy is adapted from that used on the Zeckendorf Game \cite{BEFM1, BEFM2}.

\begin{lemma}\label{lem:15once}
    In one game of the Fibonacci Quilt Game, on some fixed integer $n$, Rule (2a), $\{q_1 \land q_5 \rightarrow q_2 \land q_4\}$, can be used at most once.
\end{lemma}

This is a result of the restriction placed on Rule (2a), that it may only be used when there are no other possible moves. This is crucial in ensuring that the game terminates.

\begin{proof} The trivial game $\{1^2 \rightarrow 2\}$ shows that we do not necessarily use this rule.

Now we will consider a game where Rule (2a) has been applied once.

Let us begin before the rule has been applied. Recall that this rule may only be applied when there are no other possible moves to make. Thus, at the time the rule is applied our unordered list must contain $\{q_1 \land q_5\}$. Furthermore it cannot contain $q_2, q_3, q_4, q_6, q_8$ or $q_9$ since there is a rule in Definition \ref{def:fibqgame} that applies to each of these and $q_1$ or $q_5$ \emph{and} to use Rule (2a) no other moves may be possible. For example, if the list contained $q_3$, then the move $\{q_1 \land q_3 \rightarrow q_4\}$ could be applied, so Rule (2a) could not. A rule that could be applied before (2a) for each of these $q_i$ is shown in Figure \ref{fig:15rules}.

\begin{figure}[h!]
    \centering
    \includegraphics[width = 5in]{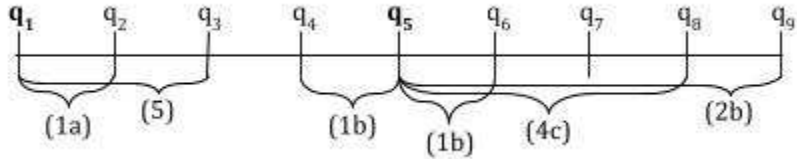}
    \caption{Rules from Definition \ref{def:fibqgame} which combine $q_1$ or $q_5$ with each of the other $q_i$.}
    \label{fig:15rules}
\end{figure}

The first term which could possibly be in the list, besides $q_1$ and $q_5$, is $q_7$, as there is no Rule to combine it with $q_1$ or $q_5$ in Figure \ref{fig:15rules}.

If $q_7 = 9$ is \emph{not} in the unordered list, then the game terminates after the rule is applied:
\[\{q_1 \land q_5 \land q_k \land \cdots \rightarrow q_2 \land q_4 \land q_k \land \cdots\},\]
where $q_k$ is the smallest possible next term, $k \geq 10$. There can be no possible moves within $\{q_k \land \cdots\}$, or else we would not have been able to apply this rule, so we are are done.

If $q_7 = 9$ is in the unordered list, then the next moves are:
\[\{q_1 \land q_5 \land q_7 \land q_\ell \land \cdots \rightarrow q_2 \land q_4 \land q_9 \land q_\ell \land \cdots \rightarrow q_1 \land q_2 \land q_8 \land q_\ell \land \cdots \rightarrow q_3 \land q_8 \land q_\ell \land \cdots\},\]
where $q_\ell$ is the smallest possible next term, $\ell \geq 12$. If $q_{12}$ is not in the list then the game is over.

If $q_{12}$ is present, we are in the situation where we have $q_{i-4} \land q_{i}$ followed by a legal decomposition. Having $q_{i}$ there are three possible next terms $q_{i+2}, q_{i+5},$ or $q_j$ for some $j > i+5$. The possible games for each of these are illustrated in the Figure \ref{fig:15tree}.

\begin{figure}[!h]
    \centering
    \includegraphics[width = 4in]{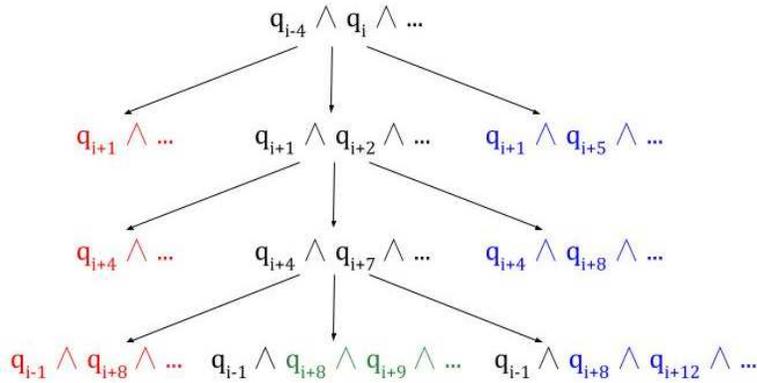}
    \caption{All possible moves, with terminating positions colored red, positions that return to the root of the tree in blue, and positions that return to $\{q_{i+1} \land q_{i+2} \land \cdots \}$ in green.     }
    \label{fig:15tree}
\end{figure}

If the next term is $q_j$ for some $j> i+5$ the game terminates immediately.

If the next term is $q_{i+5}$, then we reach $\{q_{i+1} \land q_{i+5} \land \cdots \}$. Let $k = i + 5$, then this can be rewritten as $\{q_{k-4} \land q_{k} \land \cdots \}$, and the possible games will follow the same possibilities as the root of our tree.

If the next term is $q_{i+2}$, then we reach $\{q_{i+1} \land q_{i+2} \land \cdots \}$. Again we have to consider the next possible term, there are three possibilities: $q_{i+7},q_{i+8},$ or $q_j$ for some $j > i+8$. Note that $q_{i+4}$ is not possible here although it is an acceptable distance from $q_{i+2}$ since we know $q_i$ was present and there could be no possible moves to begin with.

If the next terms are $q_{i+2}$ and $q_j$ for some $j > i+8$ then the game terminates immediately.

If the next terms are $q_{i+2}$ and $q_{i+8}$, then we reach $\{q_{i+4} \land q_{i+8} \land \cdots \}$. Let $l = i + 8$, then this can be rewritten as $\{q_{l-4} \land q_{l} \land \cdots \}$, and the possible games will follow the same possibilities as the root of our tree.

If the next terms are $q_{i+2}$ and $q_{i+7}$, then we must again consider the next possible term. If it is $q_j$ for some $j>i+12$ we reach $\{q_{i-1} \land q_{i+8} \land q_{j} \land \cdots\}$ and the game terminates. The other two possibilities are $q_{i+9}$ and $q_{i+12}$.

If the next terms are $q_{i+2}, q_{i+7}$, and $q_{i+12}$, then we reach $\{q_{i-1} \land q_{i+8} \land q_{i+12} \land \cdots \}$. Let $t = i + 12$, then this can be rewritten as $\{q_{t-4} \land q_{t} \land \cdots \}$, and the possible games will follow the same possibilities as the root of our tree.

If the next terms are $q_{i+2}, q_{i+7}$, and $q_{i+9}$, then we reach $\{q_{i-1} \land q_{i+8} \land q_{i+9} \land \cdots\}$. Let $s = i+7$, the this can be rewritten as $\{q_{s+1} \land q_{s+2} \land \cdots\}$, and the possible games will follow the same possibilities as $\{q_{i+1} \land q_{i+2} \land \cdots\}$.

Note in the last two cases we have an additional term $q_{i-1}$, however following the tree, the next smallest terms that could be created are $q_{i+11}$ and $q_{i+6}$ respectively, and there are no rules that combine $q_{i-1}$ with anything this large.

Since there are a finite number of terms in the unordered list to begin, and each time through the tree reduces this number by at least one, we know that we must terminate at some point. Throughout the tree, the smallest possible term we could create is $q_{i-1} = q_{11}$, which cannot be combined with any smaller terms, and thus we can never create another $q_5$, and the rule cannot be applied again.
\end{proof}

Knowing that we may only apply this rule at most once, we now must ensure that the game takes on only a finite number of moves before and potentially after this rule is applied. We do this by introducing a quantity which is almost a monovariant.

\begin{lemma}\label{lem:mono}
    The sum of the square roots of the indices of the $q_i$ in the unordered list on any given turn, besides Rule (2a) of $\{q_1 \land q_5 \rightarrow q_2 \land q_4\}$, is a strictly decreasing monovariant; however, Rule (2a) can be used at most once and thus this quantity is effectively a monovariant.
\end{lemma}

\begin{proof}
When considering this monovariant, we must only consider the terms directly effected by the move, since all other terms will contribute the same to the sum before and after the move. We will show the monovariant in the same order as Definition \ref{def:fibqgame} for clarity. The contribution of the terms directly effected by the move is always smaller after the move is applied.
\begin{enumerate}
    \item \textbf{Combining Consecutive Terms:}
    \begin{enumerate}
        \item $\{q_1 \land q_2 \rightarrow q_3\}$: $\sqrt{3} - \sqrt{2} - 1 \ < \ 0$
        \item $i \geq$ 2, $\{q_i \ \land \ q_{i+1} \rightarrow q_{i+3}\}$: $\sqrt{i+3} - \sqrt{i+1} - \sqrt{i}  \ < \ 0$
    \end{enumerate}
    \item \textbf{Combining $q_i$ and $q_{i+4}$}
        \begin{enumerate}
            \item $\{q_1 \ \land \ q_5 \rightarrow q_2 \ \land \ q_4 \}$: this rule is not included in this lemma.
            \item $i \geq$ 2, $\{q_i \ \land \ q_{i+4} \rightarrow q_{i+5}\}$: $\sqrt{i+5} - \sqrt{i+4} - \sqrt{i} \ <  \ 0$
        \end{enumerate}
        \item \textbf{Combining $2q_i$}
        \begin{enumerate}
            \item $\{q_1^2 \rightarrow q_2\}$: $\sqrt{2} - 2 \ < \ 0$
            \item $\{q_2^2 \rightarrow q_4\}$: $2 - 2\sqrt{2} \ < \ 0$
            \item $\{q_3^2 \rightarrow q_2 \ \land \ q_4\}$: $2 + \sqrt{2} - 2\sqrt{3} \ < \ 0$
            \item $\{q_4^2 \rightarrow q_1 \ \land \ q_6\}$: $\sqrt{6} + 1 - 4 \ <  \ 0$

            $\{q_4^2 \rightarrow q_3 \ \land \ q_5\}$: $\sqrt{5} + \sqrt{3} - 4 \ < \ 0$
            \item $\{q_5^2 \rightarrow q_1 \ \land \ q_7\}$: $\sqrt{7} + 1 - 2\sqrt{5} \ < \ 0$
            \item $\{q_6^2 \rightarrow q_2 \ \land \ q_8\}$: $\sqrt{8} + \sqrt{2} - 2\sqrt{6} \  < \ 0$

            $\{q_6^2 \rightarrow q_3 \ \land \ q_7\}$: $\sqrt{7} + \sqrt{3} - 2\sqrt{6} \ < \ 0$
            \item if $i \geq$ 7, $\{q_i^2 \rightarrow q_{i-5} \ \land \ q_{i+2}\}$: $\sqrt{i+2} + \sqrt{i-5} - 2\sqrt{i} \ < \ 0$
        \end{enumerate}
        \item \textbf{Combining $q_i$ and $q_{i+3}$}
        \begin{enumerate}
            \item $i$ = 1,2, $\{q_i \ \land \ q_{i+3} \rightarrow q_{i+4}\}$: $\sqrt{i+4} - \sqrt{i+3} - \sqrt{i} \ < \ 0$
            \item $\{q_3 \ \land \ q_6 \rightarrow q_1 \ \land q_7\}$: $\sqrt{7} + 1 - \sqrt{6} - \sqrt{3} \ < \ 0$
            \item $i$ = 4,5, $\{q_i \ \land \ q_{i+3} \rightarrow q_1 \ \land \ q_{i+4}\}$: $\sqrt{i+4} + 1 - \sqrt{i+3} - \sqrt{i} \ < \ 0$
            \item $\{q_6 \ \land \ q_9 \rightarrow q_2 \ \land \ q_{10}\}$: $\sqrt{10} + \sqrt{2} - 3 - \sqrt{6} \ < \ 0$
            \item $i \geq$ 7, $\{q_i \ \land \ q_{i+3} \rightarrow q_{i-5} \ \land \ q_{i+4}\}$: $\sqrt{i+4} + \sqrt{i-5} - \sqrt{i+3} - \sqrt{i} \ < \ 0$
        \end{enumerate}
        \item $\{q_1 \ \land \ q_3 \rightarrow q_4\}$: $2 - 1 - \sqrt{3} \ < \ 0$
    \end{enumerate}

For the values of $i$ on which this rules apply each of these is negative, thus the sum of the square roots of the indices of the terms decreases on each move and is a monovariant.
\end{proof}

With these two lemmas we can now prove Theorem \ref{thm:terminates}.

\begin{proof}
From Lemma \ref{lem:15once} we know that the rule $\{q_1 \land q_5 \rightarrow q_2 \land q_4\}$ is used either once or not at all, thus we can consider these two cases. In the first case we consider two different sub-games: the game before applying this rule and the game after. For the second case we consider the whole game.

We can see that each of these games is finite using the monovariant from Lemma \ref{lem:mono}. At the beginning of the game the sum of the square roots of the indices is $\sqrt{n}$, and with each move this value is decreasing, this means that the sum cannot be the same for two different turns and thus there will be no repeat turns. Since on each turn the unordered list is essentially a partition of $n$, of which there are finitely many, the game must terminate in finitely many moves. Similarly for a game after applying $\{q_1 \land q_5 \rightarrow q_2 \land q_4\}$, we begin with some monovariant value less than $\sqrt{n}$ and the argument continues as before.

When the game terminates it must be at an FQ-legal decomposition of $n$, since there is a rule in Definition \ref{def:fibqgame} corresponding to each illegal distance in Definition \ref{def:fqlegal}. Thus if the decomposition was not FQ-legal, there would be a rule that could be applied and the game would not be over.
\end{proof}

Now that we know the game terminates, we want to make sure that it is interesting to play; that is, that the players have choices to make and either player could win.

\begin{lemma}\label{lem:difgames}
    Given any positive integer $n$ such that $n > 3$, there are at least two distinct sequences of moves $M =\{m_i\}$ where the application of each set of moves to the initial set, denoted $M(\{q_1\}^n)$, leads to  FQ-legal decomposition of $n$.
\end{lemma}

\begin{proof}
 To show this we must only show that there are two distinct games on $n = 4$, for $n > 4$ starting with these moves would create two distinct games. There are exactly two distinct games on $n = 4$, they are
\[M_1 = \{\{q_1^2 \rightarrow q_2\},\{q_1 \land q_2 \rightarrow q_3\}, \{q_1 \land q_3 \rightarrow q_4\}\} \]
\vspace{-5mm}
\[M_2 = \{\{q_1^2 \rightarrow q_2\}, \{q_1^2 \rightarrow q_2\}, \{q_2^2 \rightarrow q_4\}\}\]

 Thus there are distinct games for $n > 3$.
\end{proof}

For $n \leq 3$ there is only one unique game. For $n = 4$ there are two unique games, and for $n = 5$ there are four games; however all games have the same length. Games begin to vary in length at $n = 6$.

\begin{corollary}\label{cor:diflength}
    For all $n > 5$, there are at least two games with different numbers of moves.  Further, there is always a game with an odd number of moves and one with an even number of moves.
\end{corollary}

\begin{proof}
There are two distinct games, one of odd length and one of even length for $n = 6$. A game of odd length on $n = 6$ is
 \[\{\{q_1^2 \rightarrow q_2\}, \{q_1 \land q_2 \rightarrow q_3\}, \{q_1 \land q_3 \rightarrow q_4\}, \{q_1 \land q_4 \rightarrow q_5\}, \{q_1 \land q_5 \rightarrow q_2 \land q_4\}\}.\]
A game of even length on $n = 6$ is
 \[\{\{q_1^2 \rightarrow q_2\}, \{q_1^2 \rightarrow q_2\}, \{q_1 \land q_2 \rightarrow q_3\}, \{q_1 \land q_3 \rightarrow q_ 4\}\}.\]

For $n \geq 7$ it is enough to show that there are two distinct games, one of odd length and one of even length for $n = 7$. For $n>7$ we know there is some $M'$ that takes $\{q_1^{n-7} \land q_6 \}$ to a FQ-legal decomposition of $n$. If the length of $M'$ is even, combine it with the even game on $n = 7$ to get an even length game, and the odd game on $n = 7$ to get an odd game, and similarly if the length of $M'$ is odd.

A game of odd length on $n = 7$ is
 \[\{\{q_1^2 \rightarrow q_2\}, \{q_1 \land q_2 \rightarrow q_3\},\{q_1^2 \rightarrow q_2\}, \{q_1 \land q_2 \rightarrow q_3\}, \{q_3^2 \rightarrow q_2 \land q_4\}, \{q_1 \land q_2 \rightarrow q_3\}, \{q_3 \land q_4 \rightarrow q_6\}\}.\]

A game of even length on $n = 7$ is
 \[\{\{q_1^2 \rightarrow q_2\}, \{q_1 \land q_2 \rightarrow q_3\}, \{q_1^2 \rightarrow q_2\}, \{q_1^2 \rightarrow q_2\}, \{q_2^2 \rightarrow q_4\}, \{q_3 \land q_4 \rightarrow q_6\}\}.\]

Thus there are game of even and odd length for all $n \geq 6$.
\end{proof}

In the next section we will explore the behavior of the game length more.

%%%%%%%%%%%%%%%%%%%%%%%%%%%%%%%%%%%%%%%%%%%%%%%%%%%%%%%%%%%%%%%%%%%%%%%%%
%%%%%%%%%%%%%%%%%%%%%%%%%%%%%%%%%%%%%%%%%%%%%%%%%%%%%%%%%%%%%%%%%%%%%%%%%
\subsection{Game Length}\label{sec:length}

For some $n$, FQ-legal decompositions are not unique. The smallest example of this is $n = 8 = 1 + 7 = 3 + 5$. From this we can define two values for $n$. Let $L(n)$ be the \textit{maximum} number of terms in an FQ-legal decomposition of $n$, and let $l(n)$ be the \textit{minimum} number of terms in an FQ-legal decomposition of $n$.
For $n = 8$, we see that $L(8) = l(8) = 2$, however they are not always equal. For example $50 = 49 + 1 = 2 + 4 + 16 + 28$, so $l(50) = 2$ but $L(50) = 4$. Theorem \ref{thm:shortest} says that the shortest possible game on $n$ is achieved in $n - L(n)$ moves.

\begin{proof}[Proof of Theorem \ref{thm:shortest}]
 Note that this is trivially true for $n = 1$. It takes $0 = 1 - 1$ moves to complete the game on 1.

Assume that the shortest game on $i$ for $1 \leq i \leq n - 1$ are achieved in $i - L(i)$ moves. Then consider the shortest possible game on $n$.

If $n$ is in the Fibonacci Quilt Sequence denote it $q_j$. One can quickly verify for $j < 5$ that the lower bound holds. For $j \geq 5$, $q_j = q_{j-2} + q_{j-3}$. To reach the right hand side it would take ($q_{j-2} - 1$) $+$ ($q_{j-3} - 1)$ $=$ $q_j - 2$ moves, using one additional move to combine $q_{j-2}, q_{j-3}$ gives us $q_j$ in $q_j - 1$ moves.

If $n$ is not in the Fibonacci Quilt Sequence then write it in an FQ-legal decomposition using the maximum possible number of terms. $n = q_{\ell_1} + q_{\ell_2} +  \cdots + q_{\ell_{L(n)}}$. To reach the right hand side it would take ($q_{\ell_1} - 1$) $+$ ($q_{\ell_2} - 1$) $ + \cdots +$ ($q_{\ell_{L(n)}} - 1)$ $=$ ($q_{\ell_1} + q_{\ell_2} + \cdots + q_{\ell_{L(n)}}$) - $L(n)$ $=$ $n - L(n)$ moves.

 To see why the game would not terminate in less moves note that every move can reduce the total number of terms in the unordered list by at most 1. Thus after $n - L(n) - 1$ moves we would still have at least $n - (n - L(n) - 1) = L(n) + 1$ terms, which cannot be an FQ-legal decomposition as $L(n)$ is the maximum.
\end{proof}

From this theorem we see that we must be able to play the game without using any of the rules which take two terms to two terms, since we must remove one term on each turn to reach the lower bound.

\begin{corollary}\label{cor:smartgame}
    It is possible, for any $n$, to play the Fibonacci Quilt Game without using Rules (2a), (3c-g), (4b-e).
\end{corollary}

We have obtained this lower bound for many values of $n$, but an algorithm to reach the lower bound for all $n$ is still unknown.

For small values of $n$ it is clear that the lower bound will not be reached in a large number of possible games. To better understand the length of an average game we used Mathematica code to simulate completely random games, where every possible move on each turn was equally likely. We then looked at the distribution of random games as $n$ increased, leading us to Conjecture \ref{con:fqgauss}, that the distribution of these random games will approach a Gaussian curve as $n$ approaches infinity.

We ran 10,000 simulations of a random game, and plotted the distribution of game lengths. For small values of $n$, the Gaussian does not fit as well. Figure \ref{fig:fqgauss20} shows the distribution of random games for $n = 20$.

\begin{figure}[!h]
    \centering
    \includegraphics[width = 3.5in]{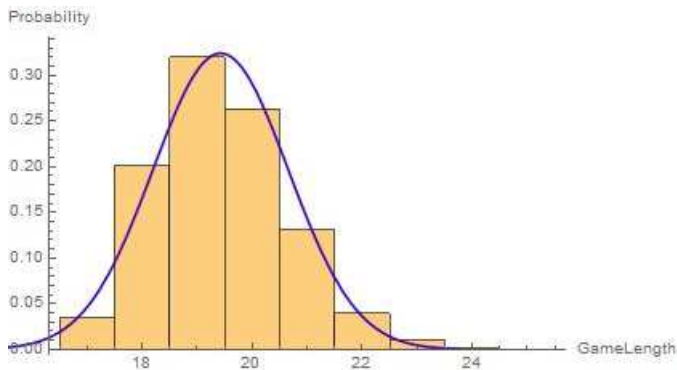}
    \caption{The distribution of game lengths of 10,000 random games on $n = 20$}
    \label{fig:fqgauss20}
\end{figure}

As we increase $n$ to 200 in Figure \ref{fig:fqgauss200}, we see that the Gaussian curve fits better.

\begin{figure}[!h]
    \centering
    \includegraphics[width = 3.5in]{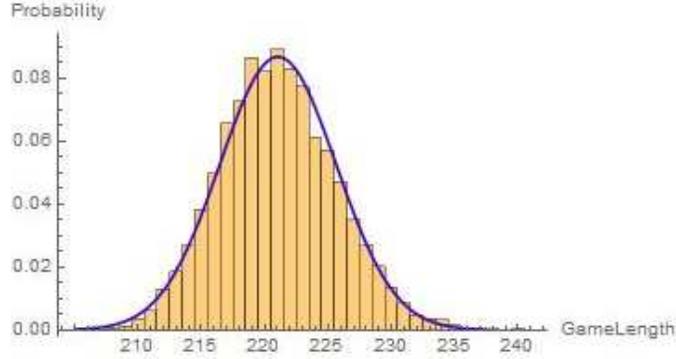}
    \caption{The distribution of game lengths of 10,000 random games on $n = 200$}
    \label{fig:fqgauss200}
\end{figure}

We also looked at the moments of these distributions compared to those with the same mean and standard deviation, with the differences of these values are shown in Figure \ref{fig:percents} (note that since we are using the best fit Gaussian, there is no error in the mean or second moment).

\begin{figure}[!h]
    \begin{tabular}{ c | c | c | c }
         $n$ & $2$\textsuperscript{nd} Moment Difference & $4$\textsuperscript{th} Moment Difference & $6$\textsuperscript{th} Moment Difference  \\ \hline
         20&0&0.044176&0.219217 \\
         60&0&0.009249&0.046575 \\
         200&0&0.000008& 0.004052
    \end{tabular}
    \caption{The percent difference between the moments of the distribution and the moments of the Gaussian curve with the same mean and standard deviation.}
    \label{fig:percents}
\end{figure}

Thus a Gaussian curve appears to fit the randomly simulated games well. From this we can also see that in a random game either player has an equal chance of winning, so the game is fair.

%%%%%%%%%%%%%%%%%%%%%%%%%%%%%%%%%%%%%%%%%%%%%%%%%%%%%%%%%%%%%%%%%%%%%%%%%%%%%%%%%%%%%%%%%%%%%%%%%%%%%%%%%%%%%%%%%%%%%%%%%%%%%%%%%%%%%%%%%%%%%%%%%%%%%%%%%%
%%%%%%%%%%%%%%%%%%%%%%%%%%%%%%%%%%%%%%%%%%%%%%%%%%%%%%%%%%%%%%%%%%%%%%%%%%%%%%%%%%%%%%%%%%%%%%%%%%%%%%%%%%%%%%%%%%%%%%%%%%%%%%%%%%%%%%%%%%%%%%%%%%%%%%%%%%
%%%%%%%%%%%%%%%%%%%%%%%%%%%%%%%%%%%%%%%%%%%%%%%%%%%%%%%%%%%%%%%%%%%%%%%%%%%%%%%%%%%%%%%%%%%%%%%%%%%%%%%%%%%%%%%%%%%%%%%%%%%%%%%%%%%%%%%%%%%%%%%%%%%%%%%%%%
%%%%%%%%%%%%%%%%%%%%%%%%%%%%%%%%%%%%%%%%%%%%%%%%%%%%%%%%%%%%%%%%%%%%%%%%%%%%%%%%%%%%%%%%%%%%%%%%%%%%%%%%%%%%%%%%%%%%%%%%%%%%%%%%%%%%%%%%%%%%%%%%%%%%%%%%%%

\section{Future work}

There are many questions about this game that can still be asked. It is known that if $n \neq 2$ then Player 2 has a winning strategy in the original Zeckendorf game. The proof techniques do not easily generalize to the Fibonacci Quilt Game due to the odd behavior of the quilt at its center, which necessitates a significantly larger set of strategies to investigate. Does Player 2 still have a winning strategy for the Fibonacci Quilt Game? If so, what is it? Note we do not know the answer to the second question for the original game; the proof that Player 2 has a winning strategy is non-constructive.

Other questions arise from bounds on game length. Is there one algorithm that reaches the lower bound for all values of $n$? Is there a reasonable upper bound on the length of a game? Simulations have never given a game of length close to or longer than $2n$, and numerical exploration of small $n$ (up to 120) suggest that the mean and maximum length grow linearly.

Another great question, asked by Dylan King at a presentation of this work, relates to $L(n)$ and $l(n)$. We give an example where $L(n)-l(n) = 0$, and another where $L(n) - l(n) = 2$, but can the distance between these two values grow arbitrarily large?

Lastly, like the Zeckendorf Game, this game has been constructed for two players, but one could also study how the behavior of this game changes if it was constructed to be played by more people at once. Who has a winning strategy (as a function of $n$ and the number of people)?

%%%%%%%%%%%%%%%%%%%%%%%%%%%%%%%%%%%%%%%%%%%%%%%%%%%%%%%%%%%%%%%%%%%%%%%%%%%%%%%%%%%%%%%%%%%%%%%%%%%%%%%%%%%%%%%%%%%%%%%%%%%%%%
%%%%%%%%%%%%%%%%%%%%%%%%%%%%%%%%%%%%%%%%%%%%%%%%%%%%%%%%%%%%%%%%%%%%%%%%%%%%%%%%%%%%%%%%%%%%%%%%%%%%%%%%%%%%%%%%%%%%%%%%%%%%%%

\ \\

\end{document}